\documentclass{article}
\usepackage{amssymb}
\usepackage{cite}
\newcommand{\R}{{{\mathbb R}}}

\newcommand{\Z}{{{\mathbb Z}}}
\newcommand{\T}{{{\mathbb T}}}

\newtheorem{ej}{Example}[section]
\newtheorem{lema}{Lemma}[section]
\newtheorem{teor}{Theorem}[section]
\newtheorem{propo}{Proposition}[section]
\newtheorem{definicion}{Definition}[section]

\def\qed{\hbox to 0pt{}\hfill$\rlap{$\sqcap$}\sqcup$\medbreak}

\def\W{$W_\Delta^{1,p}(J)$}

\title{Variational approach to second-order impulsive dynamic equations on time scales}

\author{{\sc Victoria Otero--Espinar$^1$\footnote{Corresponding author, E-mail: mvictoria.otero@usc.es} \footnote{The first author was partially supported by FEDER and Ministerio de Educaci\'on y Ciencia, Spain, project MTM2010-15314.}} and {\sc Tania Pernas--Casta\~no$^{1,2}$} \\
{$^1$ Departamento de
An\'alise Matem\'atica,}\\ 
{Universidade de Santiago de Compostela, 15782 Santiago de Compostela,
}\\{Galicia, Spain}\\
{$^2$ Instituto de Ciencias Matem\'aticas (CSIC, UAM, UC3M, UCM),}\\ 
{ 28049 Madrid, Spain
}\\
E--mail: mvictoria.otero@usc.es,  tania.pernas@icmat.es. }

\begin{document}
\date{}
\maketitle

\begin{abstract}
The aim of this paper is to employ variational techniques and critical point theory to prove some conditions for the existence of solutions to nonlinear impulsive dynamic equation with homogeneous Dirichlet boundary conditions. Also we will be interested in the solutions of the impulsive nonlinear problem with linear derivative dependence satisfying an impulsive condition.
\end{abstract}



\section{Introduction}

This paper is concerned with the existence of solutions of second order impulsive
dynamic equations on time scales. More precisely, we consider the following boundary value problem:
$$
(P)\left\{\begin{array}{lcl}
-u^{\Delta \Delta}(t)+\lambda u^{\sigma}(t)=f(t,u^{\sigma}(t));& & \textrm{$\Delta$-a.e.}\quad t\in [0,T]_{\T}^{\kappa^{2}},\\
\noalign{\bigskip} u(0)=0=u(T),\\
\noalign{\bigskip} u^{\Delta}(t_{j}^{+})-u^{\Delta}(t_{j}^{-})=I_{j}(u(t_{j}^{-})),& & j=1,2,\cdots,p.
\end{array}
\right .
$$
Where the impulsive points $t_{j}\in J$ are right-dense points in an arbitrary time scale $\T$, with $t_{0}=0<t_{1}<t_{2}<\cdots,t_{p}<t_{p+1}=T$.
Here $f:[0,T]_{\T}\times\R\to\R$ and $I_{j}:\R\to\R$, $j=1,\cdots,p$, are continuous functions.

It is well known that the theory of impulsive dynamic equations provides a natural framework for mathematical modeling of many real world phenomena.  The impulsive effects exist widely in many evolution processes in which their states are changed abruptly at certain moments of time.

Applications of impulsive dynamic equations arise in biology (biological phenomena involving thresholds), medicine (bursting rhythm models), pharmacokinetics, mechanics, engineering, chaos theory,... As a consequence, there has been a significant development in impulse theory in recent years. 
We can see some general and recent works on the theory of impulsive differential equations; see \cite{J,kkr,gg,nr,erb,xnl,xra,yra,mar} and the references therein.

For a second order dynamic equation, we usually considers impulses in the position and the velocity. 
However, in the motion of spacecraft one has to consider  instantaneous impulses depending on the position that result in jump discontinuities in velocity, but with no change in position. The impulses only on the velocity, occur also in impulsive mechanics. An impulsive problem with impulses in the derivative is considered in \cite{afr}.

Moreover, we will be interested in the solutions of the impulsive nonlinear problem in time scale with derivative dependence satisfying an impulsive condition. We can see, for example, recent works on the theory of impulsive differential equations in \cite{J,gg,xnl,yra,su}.

There have been several approaches to study solutions of impulsive dynamics equations on time scales, such as the method of lower and upper solutions, fixed-point theory \cite{bno,cw,k}.  Sobolev spaces of functions on time scales, which were first introduced in \cite{aopv1} opened a very fruitful new approach in the study of dynamic equations on time scales: the use of variational methods in the context of the boundary value problems on time scales (see \cite{aopv2,aopv3}) or in second order hamiltonian systems \cite{zhou}. Moreover, the study of the existence and multiplicity of solutions for impulsive dynamic equations on time scales has also been done by means of varational method(see, for example, \cite{zwl,df}).

The aim of this paper is to use variational techniques and critical point theory to derive the existence of multiple solutions to $(P)$; we refer the reader to \cite{bp,cie,sla1,sla2} for a broad introduction to dynamic equations on time scales and to \cite{r,ma} for variational methods and critical point theory.

The paper is organized as follows. In Section 2 we gather together essential properties about Sobolev spaces on time scales proved in  \cite{aopv1,aopv4,cre} which one needs to read this paper.

The goal of Section 3 is to exhibit the variational formulation for the impulsive Dirichlet problem. As we will see, all these problems can be understood and solved in terms of the minimization of a functional, usually related to the energy, in an appropriate space of functions. The results presented in the part where we address the linear problem, are basic but crucial to reveal that a problem can be solved by finding the critical points of a functional. Moreover, we prove some sufficient conditions for the existence of at least one positive solution to $(P)$. 

To finish, in Section 4, we present an impulsive nonlinear problem with linear derivative dependence.
We transform the problem into an equivalent one that has no dependence on the derivative and then we prove that the problem has at least one solution.
 Also, with additional conditions in nonlinearities and impulse functions,  we can show the existence of at least two solutions by using the Mountain Pass Theorem.

\section{Preliminaries}

Let $\T$ be an arbitrary time scale.   We assume that $\mathrm{\T}$ has the topology that it inherits from the standard topology on $\R$. Assume that $a<b$ are points in ${\mathrm{\T}}$ and define the time scale interval $ [a,b]_{\mathrm{\T}}=\{t\in \mathrm{\T}:a\leq t\leq b\}$. We denote by $J^{0}=[a,b)_{\mathrm{\T}}$.

Below we set out some results proved in \cite{aopv1,aopv4} about Sobolev spaces on time scales.
\begin{definicion}
Let $p\in\bar{\R}$ be such that $p\ge 1$ and $u: J\to\bar{\R}$. We say that $u$ belongs to $W_{\Delta}^{1,p}(J)$ if and only if $u\in L_{\Delta}^{p}(J^{0})$ and there exists $g:J^{\kappa}\to\bar{\R}$ such that $g\in L_{\Delta}^{p}(J^{0})$ and
\begin{displaymath}
	\int_{J^{0}} (u\cdot \varphi^{\Delta})(s)\Delta s = - \int_{J^{0}}(g\cdot  \varphi^{\sigma})(s)\Delta s\qquad \forall \varphi \in \mathcal{C} _{0,rd}^{1} (J^{k})
	\end{displaymath}
	with
	\begin{displaymath}
	\mathcal{C}_{0,rd}^{1}(J^{k})=\{ u:J\longrightarrow\mathbb{R}: u\in\mathcal{C}_{rd}^{1}(J^{k}), u(a)=u(b)=0\},
	\end{displaymath}
	and $\mathcal{C}_{rd}^{1}(J^{\kappa})$ is the set of all continuous functions on $J$ such that they are $\Delta$-differentiable on $J^{\kappa}$ and their $\Delta$-derivatives are $rd$-continuous on $J^{\kappa}$.
\end{definicion}

\begin{teor}
Assume $p\in\bar{\R}$ and $p\ge 1$. The set \W is a Banach space with the norm defined for every $x\in$\W as
\begin{equation}
\label{norm}
\Arrowvert x\Arrowvert_{W_{\Delta}^{1,p}}:=\Arrowvert x\Arrowvert_{ L_{\Delta}^{p}}+\Arrowvert x^{\Delta}\Arrowvert_{ L_{\Delta}^{p}}.
\end{equation}

Moreover, the set $H_{\Delta}^{1}(J):= W_{\Delta}^{1,2}(J)$ is a Hilbert space  with the inner product given for every $(x,y)\in H_{\Delta}^{1}(J)\times H_{\Delta}^{1}(J)$ by
\begin{equation}
\label{pi}
(x,y)_{H_{\Delta}^{1}}:=(x,y)_{ L_{\Delta}^{2}}+(x^{\Delta},y^{\Delta})_{ L_{\Delta}^{2}}.
\end{equation}
\end{teor}

\begin{propo}\label{cj}
Assume $p\in\bar{\R}$ with $p\ge 1$, then there exist a constant $K>0$, only dependent on $b-a$, such that the inequality
\begin{displaymath}
\Arrowvert x\Arrowvert_{\mathcal{C}(J)}\le K\cdot\Arrowvert x\Arrowvert_{W_{\Delta}^{1,p}}
\end{displaymath}
holds for all $x\in$\W and hence, the immersion \W$\hookrightarrow\mathcal{C}(J)$ is continuous.
\end{propo}

\begin{definicion}
Let $p\in\R$ be such that $p\ge 1$, define the set $W_{0,\Delta}^{1,p}(J)$ as the closure of the set $\mathcal{C}_{0,rd}^{1}(J^{\kappa})$ in \W. We define $H_{0,\Delta}^{1}(J):=W_{0,\Delta}^{1,2}(J)$.
\end{definicion}

The spaces $W_{0,\Delta}^{1,p}(J)$ and $H_{0,\Delta}^{1}(J)$ are endowed with the norm induced by $\Arrowvert\cdot\Arrowvert_{W_{\Delta}^{1,p}}$\-, defined in (\ref{norm}), and the inner product induced by $(\cdot,\cdot)_{H_{\Delta}^{1}}$, defined in (\ref{pi}).These spaces satisfy the following properties:
\begin{propo}{(Poincare's inequality)}
Let $p\in\R$ be such that $p\ge 1$. Then, there exists a constant $L>0$, only dependent on $b-a$, such that
\begin{displaymath}
\Arrowvert u\Arrowvert_{W_{\Delta}^{1,p}}\le L\cdot\Arrowvert u^{\Delta}\Arrowvert_{ L_{\Delta}^{p}}\qquad\forall u\in W_{0,\Delta}^{1,p}(J).
\end{displaymath}
\end{propo}

\begin{propo}{[Corollary $3.3$ in \cite{aopv4}]}
\label{wirt}
If $u\in H_{0,\Delta}^{1}(J)$, then 
\begin{displaymath}
	\int_{a}^{b}(u^{\sigma})^{2}(t)\Delta t\le\frac{1}{\lambda_{1}}\int_{a}^{b}(u^{\Delta})^{2}(t)\Delta t
\end{displaymath}
holds, where $\lambda_{1}$ is the smallest positive eigenvalue of problem $-u^{\Delta\Delta}(t)=\lambda u^{\sigma}(t);\quad t\in J^{\kappa^{2}}$ and $u(a)=0=u(b)$. 
\end{propo}

In the Sobolev space $H_{0,\Delta}^{1}(J)$ with $a=0$ and $b=T$, consider the inner product 
\begin{displaymath}
(u,v)=\int_{0}^{T}u^{\Delta}(t)v^{\Delta}(t)\Delta t
\end{displaymath}
inducing the norm $\Arrowvert \cdot\Arrowvert$.

It is consequence of Poincare's inequality that:
\begin{equation}
\label{des1}
\Arrowvert u\Arrowvert_{H_{0,\Delta}^{1}(J)}\le\Arrowvert u\Arrowvert_{W_{\Delta}^{1,2}(J)}\le c\Arrowvert u^{\Delta}\Arrowvert_{ L_{\Delta}^{2}(J^{0})}\equiv c\Arrowvert u\Arrowvert,
\end{equation} 
and 
\begin{equation}
\label{des2}
\Arrowvert u\Arrowvert \le \Arrowvert u\Arrowvert_{W_{\Delta}^{1,2}(J)}\le 2\Arrowvert u\Arrowvert_{H_{0,\Delta}^{1}(J)}.
\end{equation}

\section{Variational formulation of $(P)$ and existence results}

Firstly, to show the variational structure underlying an impulsive dynamic equation, we consider the lineal problem
$$
(LP)\left\{\begin{array}{lcl}
-u^{\Delta \Delta}(t)+\lambda u^{\sigma}(t)=h(t),& &\textrm{$\Delta$-a.e.}\quad t\in J^{\kappa^{2}},\\
\noalign{\bigskip} u^{\Delta}(t_{j}^{+})-u^{\Delta}(t_{j}^{-})=d_{j}, &&  j=1,2,\ldots,p,\\
\noalign{\bigskip} u(0)= u(T)=0.
\end{array}
\right.
$$
where we consider $J$ with $a=0$ and $b=T$ and $d_{j}$, $j=1,\cdots,p$, are fixed constants.

Suppose that $u\in\mathcal{C}_{rd}(J)$ is such that $u(0)=0=u(T)$. Moreover assume that for every $j=0,1,\cdots,p$, $u_{j}:=u\mid_{(t_{j},t_{j+1})}$ is such that $u_{j}\in H_{\Delta}^{2}(t_{j},t_{j+1})$.
\begin{definicion}
We say that $u$ is a \emph{classical solution} of $(LP)$ if the limits $u^{\Delta}(t_{j}^{+})$ and $u^{\Delta}(t_{j}^{-})$ exist for every $j=1,2,\ldots,p$ and it satisfies the equation on $(LP)$ for $\Delta$-almost everywhere ($\Delta$-a.e.) $t\in J^{\kappa^{2}}$.
\end{definicion}

	Take $v\in H_{0,\Delta}^{1}(J)$, multiply the equation by $v^{\sigma}$ and integrate between $0$ and $T$:
	\begin{displaymath}
	-\int_{0}^{T}u^{\Delta\Delta}v^{\sigma}+\lambda\int_{0}^{T}u^{\sigma}v^{\sigma}=\int_{0}^{T}hv^{\sigma}.
	\end{displaymath}
	
	Taking account that $v(0)=0=v(T)$, integrating by parts:
	\begin{displaymath}
	-\int_{0}^{T}u^{\Delta\Delta}v^{\sigma}=-\sum_{j=0}^{p}\int_{t_{j}}^{t_{j+1}}u^{\Delta\Delta}v^{\sigma}=\sum_{j=0}^{p}\left\lbrack -(u^{\Delta}v)\mid_{t_{j}^{+}}^{t_{j+1}^{-}}+\int_{t_{j}}^{t_{j+1}}u^{\Delta}v^{\Delta}\right\rbrack =
\end{displaymath}
	\begin{displaymath}
	=\sum_{j=1}^{p}d_{j}v(t_{j})+\int_{0}^{T}u^{\Delta}v^{\Delta}.
\end{displaymath}
	Hence,
	\begin{displaymath}
	\int_{0}^{T}u^{\Delta}v^{\Delta}+\lambda\int_{0}^{T}u^{\sigma}v^{\sigma}=\int_{0}^{T}hv^{\sigma}-\sum_{j=1}^{p}d_{j}v(t_{j}).
\end{displaymath}	
	We define the bilinear form $a:H_{0,\Delta}^{1}(J)\times H_{0,\Delta}^{1}(J)\longrightarrow\mathbb{R}$ by
	\begin{equation}
	\label{a}
	a(u,v)=\int_{0}^{T}u^{\Delta}v^{\Delta}+\lambda\int_{0}^{T}u^{\sigma}v^{\sigma},
	\end{equation}
	and the linear operator $l:H_{0,\Delta}^{1}(J)\longrightarrow\mathbb{R}$ by
	\begin{equation}
	\label{l}
	 l(v)=\int_{0}^{T}hv^{\sigma}-\sum_{j=1}^{p}d_{j}v(t_{j}).
	 \end{equation}

	 Thus, the concept of \emph{weak solution} for the impulsive problem $(LP)$ is a function $u\in H_{0,\Delta}^{1}(J)$ such that $a(u,v)=l(v)$ is valid for any $v\in H_{0,\Delta}^{1}(J)$.
	 
   We can prove that $a$ defined by (\ref{a}) and $l$ defined by (\ref{l}) are continuous, and, from Proposition $\ref{wirt}$, that $a$ is coercive if $\lambda>-\lambda_{1}$.
	 
	 Consider $\varphi:H_{0,\Delta}^{1}(J)\to\R$ defined by
	 \begin{equation}
	 \label{varphi} \varphi(v)=\frac{1}{2}a(v,v)-l(v)=\frac{1}{2}\int_{0}^{T}(v^{\Delta})^{2}+\frac{\lambda}{2}\int_{0}^{T}(v^{\sigma})^{2}-\int_{0}^{T}hv^{\sigma}+\sum_{j=1}^p d_{j}v(t_{j}).
	 \end{equation}

We can deduce the following regularity properties which allow us to assert that the solutions to $(LP)$ are precisely the critical points of $\varphi$.

\begin{lema} The following statements are valid:
\begin{enumerate}
\item $\varphi$ is differentiable at any $u\in H_{0,\Delta}^{1}(J)$ and 
	 \begin{displaymath} (\varphi'(u),v)=\int_{0}^{T}u^{\Delta}v^{\Delta}+\lambda\int_{0}^{T}u^{\sigma}v^{\sigma}-\int_{0}^{T}hv^{\sigma}+\sum_{j=1}^{p}d_{j}v(t_{j})=a(u,v)-l(v).
	 \end{displaymath}
\item If $u\in H_{0,\Delta}^{1}(J)$ is a critical point of $\varphi$ defined by (\ref{varphi}), then $u$ is a weak solution of the impulsive problem $(LP)$.
\end{enumerate}
\end{lema}

We will use the following result in linear functional analysis which ensures the existence of a critical point of $\varphi$.

\begin{teor}{(Lax-Milgram Theorem)}
Let $H$ be a Hilbert space and $a:H\times H\to\R$ a bounded bilinear form. If $a$ is coercive, i.e., there exists $\alpha>0$ such that $a(u,u)\ge\alpha\Arrowvert u\Arrowvert^{2}$ for every $u\in H$, then for any $\sigma\in H'$ (the conjugate space of H) there exists a unique $u\in H$ such that
\begin{displaymath}
a(u,v)=\langle\varphi, v\rangle\quad v\in H.
\end{displaymath}
Moreover, if $a$ is also symmetric, then the functional $J:H\to\R$ defined by
\begin{displaymath}
J(v)=\frac{1}{2}a(u,v)-\langle\varphi, v\rangle,
\end{displaymath}
attains its minimum at $u$.
\end{teor}

By the Lax-Milgram theorem  we obtain the following result

\begin{teor}
If $\lambda>-\lambda_{1}$ then the problem $(LP)$ has a weak solution $u\in H_{0,\Delta}^{1}(J)$ for any $h\in L_{\Delta}^{2}(J^{0})$. Moreover, $u\in H_{\Delta}^{2}(J)$ and $u$ is a classical solution and $u$ minimizes the functional (\ref{varphi}) and hence it is a critical point of $\varphi$.
\end{teor}

\subsection{Impulsive nonlinear problem}

We consider the nonlinear Dirichlet problem 
$$
(P)\left\{\begin{array}{lcl}
-u^{\Delta \Delta}(t)+\lambda u^{\sigma}(t)=f(t,u^{\sigma}(t));& & \textrm{$\Delta$-a.e.}\quad t\in J^{\kappa^{2}},\\
\noalign{\bigskip} u^{\Delta}(t_{j}^{+})-u^{\Delta}(t_{j}^{-})=I_{j}(u(t_{j}^{-})),& & j=1,2,\cdots,p,\\
\noalign{\bigskip} u(0)=0=u(T).
\end{array}
\right .
$$
 We assume that $\lambda>-\lambda_{1}$.

	A weak solution of $(P)$ is a function $u\in H_{0,\Delta}^{1}(J)$ such that
	\begin{displaymath}	\int_{0}^{T}u^{\Delta}v^{\Delta}+\lambda\int_{0}^{T}u^{\sigma}v^{\sigma}=-\sum_{j=1}^{p}I_{j}(u(t_{j}^{-}))v(t_{j})+\int_{0}^{T}f(t,u^{\sigma}(t))v^{\sigma}(t)\Delta t,
	\end{displaymath}
	for every $v\in H_{0,\Delta}^{1}(J)$.
	
	We now consider the functional
	\begin{equation}
	\label{phi}	\begin{array}{lcl} 
	\displaystyle  \varphi(u)=\frac{1}{2}a(u,u)-l(u)&=&\displaystyle  \frac{1}{2}\int_{0}^{T}(u^{\Delta})^{2}+\frac{\lambda}{2}\int_{0}^{T}(u^{\sigma})^{2}+\sum_{j=1}^{p}\int_{0}^{u(t_{j})}I_{j}(t)dt\\
	\\
	& &\displaystyle  -\int_{0}^{T}F(t,u^{\sigma}(t))\Delta t,
	\end{array}
	\end{equation}
	
	where $F(t,u):=\int_{0}^{u}f(t,x)dx$.
	
	One can deduce, from the properties of $H$, $f$ and $I_j$, the following regularity properties of $\varphi$. 
	\begin{propo}
	The functional $\varphi$ defined by $(\ref{phi})$ is continuous, differentiable, and weakly lower semi-continuous. Moreover, the critical points of $\varphi$ are weak solutions of $(P)$.
	\end{propo}
	\begin{teor}
	Suppose that $f$ is bounded and that the impulsive functions $I_{j}$ are bounded. Then there is a critical point of $\varphi$, and $(P)$ has at least one solution.
	\end{teor}
	\begin{proof}
	Take $M>0$ and $M_{j}>0$, $j=1,2,\cdots,p$,  such that
	\begin{displaymath}
\arrowvert f(t,u)\arrowvert\le M \quad \forall (t,u)\in\lbrack 0,T\rbrack_{\mathbb{T}}\times\mathbb{R},
\end{displaymath}
and
\begin{displaymath}
\arrowvert I_{j}(u)\arrowvert\le M_{j} \quad \forall u\in\mathbb{R}, \ j=1,2,\cdots,p.
\end{displaymath}

	Using that $\lambda>-\lambda_{1}$ there exists $\alpha>0$ sucht that for any $u\in H_{0,\Delta}^{1}(J)$
	\begin{displaymath}\begin{array}{lcl} 
	\displaystyle\varphi(u)&\ge & \displaystyle \frac{\alpha}{2}\Arrowvert u\Arrowvert^{2}+\sum_{j=1}^{p}\int_{0}^{u(t_{j})}I_{j}(t)dt-\int_{0}^{T}F(t,u^{\sigma}(t))\Delta t\ge \\

&\ge&\displaystyle \frac{\alpha}{2}\Arrowvert u\Arrowvert^{2}-\sum_{j=1}^{p}M_{j}\arrowvert u(t_{j})\arrowvert - M\int_{0}^{T}\arrowvert u^{\sigma}(t)\arrowvert\Delta t.
	\end{array} 
	\end{displaymath}

Thus, using the proposition \ref{cj}, (\ref{des1}) and $m=\max_{j=1,\cdots,p}\{M,M_{j}\}$ we have
	
	\begin{displaymath}\begin{array}{lcl} 
	
\displaystyle \varphi(u) &\ge& \displaystyle\frac{\alpha}{2}\Arrowvert u\Arrowvert^{2}-m \left(\sum_{j=1}^{p}\arrowvert u(t_{j})\arrowvert - \int_{0}^{T}\arrowvert u^{\sigma}(t)\arrowvert\Delta t \right)\\
\\
&\ge& \displaystyle \frac{\alpha}{2}\Arrowvert u\Arrowvert^{2}-m \left(p\Arrowvert u\Arrowvert_{\mathcal{C}(J)} + T\Arrowvert u\Arrowvert_{\mathcal{C}(J)}\right)\equiv\frac{\alpha}{2}\Arrowvert u\Arrowvert^{2}-m \rho\Arrowvert u\Arrowvert_{\mathcal{C}(J)}\\
\\
 &\ge& \displaystyle\frac{\alpha}{2}\Arrowvert u\Arrowvert^{2}-m\rho K\Arrowvert u\Arrowvert_{H_{0,\Delta}^{1}}\ge\frac{\alpha}{2}\Arrowvert u\Arrowvert^{2}-m\rho Kc\Arrowvert u\Arrowvert
	\end{array} 
	\end{displaymath}

where $\rho=p+T$.

This implies that $\lim_{\Arrowvert u\Arrowvert\to\infty}{\varphi(u)}=+\infty$, and $\varphi$ is coercive. Hence, (Th. 1.1 of \cite{ma}) $\varphi$ has a minimum, which is a critical point of $\varphi$. \qed
	\end{proof}
	
\begin{teor}
Suppose that $f$ is sublinear and the impulsive functions $I_{j}$ have sublinear growth. Then there is a critical point of $\varphi$ and $(P)$ has at least one solution.
\end{teor}	
\begin{proof}
Let $a$, $b$, $a_{j}$, $b_{j}>0$, and $\gamma,\ \gamma_{j}\in\lbrack 0,1)$, $j=1,2,\cdots,p$, such that
	\begin{displaymath}
	\arrowvert f(t,u)\arrowvert\le a+b\arrowvert u\arrowvert^{\gamma}\quad and \quad \arrowvert I_{j}(u)\arrowvert\le a_{j}+b_{j}\arrowvert u\arrowvert^{\gamma_{j}} \quad \forall t\in\lbrack 0,T\rbrack_{\mathbb{T}},  u\in\mathbb{R}.
	\end{displaymath}
		
	Again using that $\lambda>-\lambda_{1}$, the proposition \ref{cj}, (\ref{des1}) and $m=\max_{j=1,\cdots,p}\{a,a_{j}\}$, $\widetilde{m}=\max_{j=1,\cdots,p}\{b,b_{j}\}$ we have
	\begin{displaymath}
	\varphi(u)\ge\frac{\alpha}{2}\Arrowvert u\Arrowvert^{2}-\beta\Arrowvert u\Arrowvert-\delta\Arrowvert u\Arrowvert^{\gamma+1}
	\end{displaymath}
	where $\beta=m\rho Kc$ and $\delta=\widetilde{m}\rho K^{\gamma+1}c^{\gamma+1}$.
	
	Since $\gamma+1<2$, then $\lim_{\Arrowvert u\Arrowvert\to\infty}\varphi(u)=+\infty$ for every $u\in H_{0,\Delta}^{1}(J)$.
\end{proof}

\section{Impulsive nonlinear problem with linear derivative dependence}

	Consider the following problem
\begin{displaymath}
(NP)\left\{\begin{array}{lcl}
 -u^{\Delta \Delta}(t)+g(t)u^{\Delta}(\sigma(t))+\lambda u^{\sigma}(t)=f(t,u^{\sigma}(t));& &\textrm{$\Delta$-a.e.}\quad t\in J^{\kappa^{2}}\\
\noalign{\bigskip}  -(u^{\Delta}(t_{j}^{+})-u^{\Delta}(t_{j}^{-}))=I_{j}(u(t_{j})),& &j=1,2,\cdots,p\\
\noalign{\bigskip}  u(0)=0=u(T).
\end{array}
\right .
\end{displaymath}
where $f$ and $I_{j}$, $j=1,\cdots,p$ are continuous and $g$ is continuous and regressive.

We assume that $\lambda>-m\lambda_{1}/M$. Here, $m=\min_{t\in J}e_{g}(t,0)$, $M=\max_{t\in J}e_{g}(t,0)$ where $e_{g}(t,0)$ is the exponential function. Note that, as g is regressive, $e_{g}(\cdot,0)$ is the solution of the problem
$$y^{\Delta}= g(t)y, \quad y(0) = 1.$$

We transform the problem (NP) into the following equivalent form:
\begin{displaymath}
(NPE)\left\{ \begin{array}{lcl}
-(e_{g}(t,0)u^{\Delta}(t))^{\Delta}+\lambda e_{g}(t,0)u^{\sigma}(t)=e_{g}(t,0)f(t,u^{\sigma}(t))& &\textrm{$\Delta$-a.e.}\quad t\in J^{\kappa^{2}}\\ 
\noalign{\bigskip} -(u^{\Delta}(t_{j}^{+})-u^{\Delta}(t_{j}^{-}))= I_{j}(u(t_{j})),& & j=1,2,\ldots,p\\
\noalign{\bigskip} u(0)=u(T)=0
\end{array} \right.
\end{displaymath}

Obviously, the solutions of (NPE) are solutions of (NP). Consider the Hilbert space $H_{0,\Delta}^{1}(J)$ with the inner product:
\begin{displaymath}
(u,v)=\int_{0}^{T}e_{g}(t,0)u^{\Delta}(t)v^{\Delta}(t)\Delta t,
\end{displaymath}
and the norm induced 
\begin{displaymath}
\Arrowvert u\Arrowvert=\left(\int_{0}^{T}e_{g}(t,0)\arrowvert u^{\Delta}(t)\arrowvert^{2}\Delta t \right)^{\frac{1}{2}}.
\end{displaymath}

	 A weak solution of (NPE) is a function $u\in H_{0,\Delta}^{1}(J)$ such that
	 
$$\displaystyle\int_{0}^{T}e_{g}(t,0)u^{\Delta}(t)v^{\Delta}(t)\Delta t+\lambda\int_{0}^{T}e_{g}(t,0)u^{\sigma}(t)v^{\sigma}(t)\Delta t
	$$ $$=\displaystyle\sum_{j=1}^{p}e_{g}(t_{j},0)I_{j}(u(t_{j}))v(t_{j})+\int_{0}^{T}e_{g}(t,0)f(t,u^{\sigma}(t))v^{\sigma}(t)\Delta t.$$

Hence, a weak solution of (NP) is a critical point of the following functional:
	\begin{equation}
	\label{psi}
	\psi(u)=\frac{1}{2}A(u,u)-\sum_{j=1}^{p}e_{g}(t_{j},0)\int_{0}^{u(t_{j})}I_{j}(t)dt - \int_{0}^{T}e_{g}(t,0)F(t,u^{\sigma}(t))\Delta t,
	\end{equation}
where $$F(t,u)=\int_{0}^{u}f(t,\xi)d\xi,$$ 
	and $$	A(u,u)=\int_{0}^{T}e_{g}(t,0)u^{\Delta}(t)v^{\Delta}(t)\Delta t + \lambda\int_{0}^{T}e_{g}(t,0)u^{\sigma}(t)v^{\sigma}(t)\Delta t.$$
	
	It is evident that $A$ is bilinear, continuous and symmetric.
	
	\begin{lema}{(Theorem $38.A$ of \cite{zei})}\label{zeim}
	For the functional  $F:M\subset X\to\R$ with $M$ not empty, $\min_{u\in M}F(u)=a$ has a solutions in case the following hold:
	\begin{itemize}
	\item [$(i)$] $X$ is a reflexive Banach space.
	\item [$(ii)$] $M$ is bounded and weak sequentially closed.
	\item [$(iii)$] $\varphi$ is sequentially lower semi-continuous on $M$
	\end{itemize}
	\end{lema}
	
	\begin{lema}{(Analogous to lemma 2.2 of \cite{erb})}
	There exist constants $\beta>\alpha>0$ such that
	\begin{displaymath}
	\alpha\Arrowvert u\Arrowvert^{2}\le A(u,u)\le\beta\Arrowvert u\Arrowvert^{2},\quad u\in H_{0,\Delta}^{1}(J).
	\end{displaymath} 
	\end{lema}
	\begin{proof}
	In fact, by Poincare's inequality, if $\lambda\ge 0$, we can take $\alpha=1$, $\beta=1+\frac{\lambda M}{\lambda_{1}m}$; if $\frac{-m\lambda_{1}}{M}<\lambda<0$, then we can take $\alpha=1+\frac{\lambda M}{\lambda_{1}m}$ and $\beta=1$.\qed
	\end{proof}
	
	\begin{lema}\label{max}
	If $u\in H_{0,\Delta}^{1}(J)$, then there exist a constant $\delta>0$, such that $\Arrowvert u\Arrowvert_{0}\le\delta\Arrowvert u\Arrowvert$, where
	\begin{displaymath}
	\displaystyle \Arrowvert u\Arrowvert_{0}=\max_{t\in\lbrack 0,T\rbrack_{\mathbb{T}}}\arrowvert u(t)\arrowvert.
	\end{displaymath}
	\end{lema}
	\begin{proof}
	The result is followed by the following inequalities:
	$$
	\begin{array}{lcl}
	\displaystyle\arrowvert u(t)\arrowvert&\le&\int_{0}^{T}\arrowvert u^{\Delta}(s)\arrowvert\Delta s\le \left(\int_{0}^{T}\frac{1}{e_{g}(t,0)}\right)^{\frac{1}{2}}\left(\int_{0}^{T}e_{g}(t,0)\arrowvert u^{\Delta}(s)\arrowvert^{2}\Delta s\right)^{\frac{1}{2}}\\
\displaystyle&\le&\sqrt{\frac{T}{m}}\Arrowvert u\Arrowvert=\delta\Arrowvert u\Arrowvert.
	\end{array}
$$ \qed
	\end{proof}
	
	\begin{lema}\label{sci}
	The functional $\psi$ defined by (\ref{psi}) is continuous, continuously differentiable and weakly lower semi-continuous.
	\end{lema}
	\begin{teor}
	Suppose that $\lambda>\frac{-m\lambda_{1}}{M}$, $f$ and $I_{j}$ are bounded, $j=1,2,\cdots,p$, then the (NP) has at least one solution.
	\end{teor}
	\begin{proof}
	Take $B>0$ and $B_{j}>0$, $j=1,\cdots,p$, such that
	$$\arrowvert f(t,u)\arrowvert\le B,\quad\forall(t,u)\in\lbrack 0,T\rbrack_{\mathbb{T}}\times\mathbb{R},$$
	$$\arrowvert I_{j}(u)\arrowvert\le B_{j},\quad\forall u\in\mathbb{R},j=1,2,\cdots,p.$$
	
	For any $u\in H_{0,\Delta}^{1}(J)$, using lemma \ref{max} and proposition \ref{wirt},
	$$\psi(u)\ge\frac{\alpha}{2}\Arrowvert u\Arrowvert^{2}-\left(M\delta \sum_{j=1}^{p}B_{j}+B\ M\sqrt{\frac{T}{m\lambda_{1}}}\right)\Arrowvert u\Arrowvert.$$	
	This implies that $\lim_{\Arrowvert u\Arrowvert\to\infty}\psi(u)=+\infty$, and $\psi$ is coercive. Hence, $\psi$ has a minimum, which is a critical point of $\psi$.\qed
	\end{proof}

We will apply the Mountain Pass Theorem in order to obtain at least two critical point of $\psi$.

 Suppose that $X$ is a Banach space (in particular a Hilbert space) and $\phi:X\to\R$ differentiable and $c\in\R$. We say that $\phi$ satisfies the Palais-Smale condition if every bounded sequence $\{u_{k}\}$ in the space $X$ such that $lim_{k\to\infty}\phi'(u_{k})=0$ contains a convergent subsequence. 

\begin{teor}{(Mountain Pass Theorem)}\label{mpt}
Let $\phi\in\mathcal{C}^{1}$ such that satisfies the Palais-Smale condition. Assume that there exists $u_{0}, u_{1}\in X$ and a bounded neighbourhood $\Omega$ of $u_{0}$ such that $u_{1}\notin\Omega$ and
\begin{displaymath}
inf\{\phi(u):u\in\partial\Omega\}>\max\{\phi(u_{0}),\phi(u_{1})\}
\end{displaymath}
Then there exists a critical point $u^{*}$ of $\phi$.

\end{teor}

	\begin{teor}\label{teoprin}
	Suppose that $\lambda>\frac{-m\lambda_{1}}{M}$, then the problem (NP) has at least two solutions if the following conditions hold:
	\begin{itemize}
	\item [$(H_{1})$] There exist constants $\eta>2$ and $\gamma>0$, such that for all $(t,u)\in\lbrack 0,T\rbrack_{\mathbb{T}}\times\mathbb{R}$, $\arrowvert u\arrowvert\ge\gamma$
	\begin{displaymath}
	0<\eta F(t,u)\le uf(t,u),\quad 0<\eta\int_{0}^{u}I_{j}(\xi)d\xi\le uI_{j}(u),
	\end{displaymath} 
	where $j=1,2,\cdots,p$.
	\item [$(H_{2})$] There exists a positive $s\ge\eta$ such that $f(t,u)=o(\arrowvert u\arrowvert^{s})$ and $I_{j}(u)=o(\arrowvert u\arrowvert^{s})$ uniformly for $t\in\lbrack 0,T\rbrack_{\mathbb{T}}$ as $\arrowvert u\arrowvert\to\infty$, $j=1,2,\cdots,p$.
	\item [$(H_{3})$]$f(t,u)=o(\arrowvert u\arrowvert)$ and $I_{j}(u)=o(\arrowvert u\arrowvert)$ uniformly for $t\in\lbrack 0,T\rbrack_{\mathbb{T}}$ as $\arrowvert u\arrowvert\to 0$, $j=1,2,\cdots,p$.
\end{itemize}
	\end{teor}
	\begin{proof}
	
	From $(H_{2})$, $(H_{3})$ and the continuities of $f$ and $I_{j}$, is easy to see that for any $ \varepsilon>0$ and $(t,u)\in\lbrack 0,T\rbrack_{\mathbb{T}}\times\mathbb{R}$, there exist $C_{1}( \varepsilon)>0$ y $C_{1j}>0$ such that
	\begin{displaymath}
	\arrowvert f(t,u)\arrowvert\le \varepsilon\arrowvert u\arrowvert + C_{1}( \varepsilon)\arrowvert u\arrowvert^{s},
	\end{displaymath} 
	\begin{displaymath}
	\arrowvert I_{j}(u)\arrowvert\le \varepsilon\arrowvert u\arrowvert + C_{1j}( \varepsilon)\arrowvert u\arrowvert^{s}.
	\end{displaymath}
	
	Hence, for any $ \varepsilon>0$ and $(t,u)\in\lbrack 0,T\rbrack_{\mathbb{T}}\times\mathbb{R}$, we have
	\begin{equation}
	\label{F}
	F(t,u)\ \le\ \int_{0}^{\arrowvert u\arrowvert}\left\lbrack \varepsilon\xi +C_{1}( \varepsilon)\xi^{s}\right\rbrack d\xi\ \le\ \frac{ \varepsilon}{2}\arrowvert u\arrowvert^{2}+C_{2}( \varepsilon)\arrowvert u\arrowvert^{s+1},
	\end{equation}
	
	\begin{equation}
	\label{I}
	\int_{0}^{u}I_{j}(\xi)d\xi\ \le\ \int_{0}^{\arrowvert u\arrowvert}\lbrack \varepsilon\xi +C_{1j}( \varepsilon)\xi^{s}\rbrack d\xi\ \le\ \frac{ \varepsilon}{2}\arrowvert u\arrowvert^{2}+C_{2j}( \varepsilon)\arrowvert u\arrowvert^{s+1},
	\end{equation}
	
	where $\displaystyle C_{2}( \varepsilon)=\frac{C_{1}( \varepsilon)}{s+1}$ and $\displaystyle C_{2j}( \varepsilon)=\frac{C_{1j}( \varepsilon)}{s+1}$
	
	From condition $(H_{1})$, there hold:
	\begin{displaymath}
	\frac{\eta}{u}\le\frac{f(t,u)}{F(t,u)},\ \forall u\ge\gamma,\qquad\frac{\eta}{u}\ge\frac{f(t,u)}{F(t,u)},\ \forall u\le -\gamma.
	\end{displaymath}
	
	Integrating the above two inequalities with respect to $u$ on $\lbrack \gamma,u\rbrack$ and $\lbrack u,-\gamma\rbrack$ respectively (in this case, these are integrals on $\mathbb{R}$), we have
	\begin{displaymath}
	\eta\ln\frac{u}{\gamma}\le\ln\frac{F(t,u)}{F(t,\gamma)}, \ \forall u\ge\gamma,\qquad\eta\ln\frac{\gamma}{-u}\ge\ln\frac{F(t,-\gamma)}{F(t,u)},\ \forall u\le -\gamma.
	\end{displaymath}
	That is,
	\begin{displaymath}
	F(t,u)\ \ge \ F(t,\gamma)\left(\frac{u}{\gamma}\right)^{\eta},\ \forall u\ge\gamma,\qquad F(t,u)\ge F(t,-\gamma)\left(\frac{-u}{\gamma}\right)^{\eta},\ \forall u\ge -\gamma.
	\end{displaymath}
	
	Thus there exist a constant $a_{1}>0$ such that $F(t,u)\ge a_{1}\arrowvert u\arrowvert^{\eta}$ for all $\arrowvert u\arrowvert\ge\gamma$.
	
	From the continuity of $F(t,u)$, there exist a constant $k>0$, such that
	\begin{displaymath}
	F(t,u)\ \ge \ -k\ge a_{1}\arrowvert u\arrowvert^{\eta}-a_{1}\gamma^{\eta}-k,\quad\forall\arrowvert u\arrowvert\le\gamma.
	\end{displaymath}
	
	Hence, we have
	\begin{equation}
	\label{F2}
	F(t,u)\ \ge \ a_{1}\arrowvert u\arrowvert^{\eta}-a_{2},\quad\forall (t,u)\in\lbrack 0,T\rbrack_{\mathbb{T}}\times\mathbb{R},
	\end{equation}
	where $a_{2}=a_{1}\gamma^{\eta}+k$.
	
	Similarly, there exist $a_{1j},a_{2j}>0$ such that
	\begin{equation}
	\label{I2}
	\int_{0}^{u}I_{j}(\xi)d\xi\ \ge \ a_{1j}\arrowvert u\arrowvert^{\eta}-a_{2j},\quad\forall u\in\mathbb{R}.
	\end{equation}

	Firstly, we apply Lemma \ref{zeim} to show that there exists $\rho$ such that $\psi$ has a local minimum $u_{0}\in B_{\rho}=\{u\in H_{0,\Delta}^{1}(J):\Arrowvert u\Arrowvert<\rho\}$.
	
	Since $H_{0,\Delta}^{1}(J)$ is a Hilbert space, is easy to deduce that $\bar{B_{\rho}}$ is bounded and weak sequentially closed. Lemma \ref{sci} has shown that $\psi$ is weak lower semi-continuous on $\bar{B_{\rho}}$ and, besides, $H_{0,\Delta}^{1}(J)$ is a reflexive Banach space. So by \ref{zeim} we can have this $u_{0}$ such that $\psi(u_{0})=\min\{\psi(u):u\in \bar{B_{\rho}}\}$
	
	Now we will show that $\psi(u_{0})=\min\{\psi(u):u\in\partial B_{\rho}\}$ for some $\rho=\rho_{0}$ 
	
	In fact, from (\ref{F}) y (\ref{I}):
$$	\begin{array}{lcl}
	\psi(u)&\ge&\displaystyle \frac{\alpha}{2}\Arrowvert u\Arrowvert^{2}-\sum_{j=1}^{p}e_{g}(t,0)\int_{0}^{u(t_{j})}I_{j}(t)dt-\int_{0}^{T}e_{g}(t,0)F(t,u^{\sigma}(t))\Delta t \\
	&\ge&\displaystyle \frac{\alpha}{2}\Arrowvert u\Arrowvert^{2}-\sum_{j=1}^{p}e_{g}(t,0)\left(\frac{\varepsilon}{2}\arrowvert u\arrowvert^{2}+C_{2j}( \varepsilon)\arrowvert u\arrowvert^{s+1}\right)\\
&&\displaystyle -\int_{0}^{T}e_{g}(t,0)\left(\frac{ \varepsilon}{2}\arrowvert u\arrowvert^{2}+C_{2}( \varepsilon)\arrowvert u^{\sigma}\arrowvert^{s+1}\right)
\ge \displaystyle \frac{\alpha}{2}\Arrowvert u\Arrowvert^{2}-M\ \frac{\varepsilon}{2}\ p\ \delta^{2}\Arrowvert u\Arrowvert^{2}\\
&&\displaystyle -\delta^{s+1}M\sum_{j=1}^{p}C_{2j}(\varepsilon)\Arrowvert u\Arrowvert^{s+1}-\frac{M\varepsilon}{2}\int_{0}^{T}\arrowvert u^{\sigma}(t)\arrowvert^{2}\Delta t\\
&&\displaystyle -MC_{2}(\varepsilon)\int_{0}^{T}\arrowvert u^{\sigma}(t)\arrowvert^{s+1}\Delta t
\ge \displaystyle\frac{\alpha}{2}\Arrowvert u\Arrowvert^{2}-M\ \frac{\varepsilon}{2}\ p\ \delta^{2}\Arrowvert u\Arrowvert^{2}\\
&&\displaystyle -\delta^{s+1}M\sum_{j=1}^{p}C_{2j}(\varepsilon) \ \Arrowvert u\Arrowvert^{s+1}-\displaystyle \frac{M\varepsilon}{2}\frac{1}{m\lambda_{1}}\Arrowvert u\Arrowvert^{2}\\
&&\displaystyle -MC_{2}(\varepsilon)\frac{T\ k^{s+1}\ c^{s+1}}{(\sqrt{m})^{s+1}}\Arrowvert u\Arrowvert^{s+1}.	
		\end{array}
$$
	Hence,
$$	\begin{array}{lcl}
	\psi(u)&\ge&\displaystyle\frac{\alpha-M\ \varepsilon\ (\frac{1}{m\lambda_{1}}+\delta^{2}\ p)}{2}\ \Arrowvert u\Arrowvert^{2}\\
\\
&&-M\left(\delta^{s+1}\sum_{j=1}^{p}C_{2j}(\varepsilon)+\frac{C_{2}(\varepsilon)\ T\ k^{s+1}\ c^{s+1}}{(\sqrt{m})^{s+1}}\right)\Arrowvert u\Arrowvert^{s+1}.
	\end{array}
$$
	
	We can choose,
	
$$	\varepsilon=\displaystyle \frac{\alpha}{2M\left(\frac{1}{m\lambda_{1}}+\delta^{2}p\right)},$$
$$ \displaystyle \rho_{0}=\left(\frac{\alpha}{8M\left(\frac{T\ k^{s+1}\ c^{s+1}}{(\sqrt{m})^{s+1}}\ C_{2}(\varepsilon)+\delta^{s+1}\sum_{j=1}^{p}C_{2j}(\varepsilon)\right)}\right)^{\frac{1}{s-1}}.$$

	For any $u\in\partial B_{\rho_{0}}$, $\Arrowvert u\Arrowvert=\rho_{0}$, we have $\psi(u)\ge\frac{\alpha}{8\rho_{0}^{2}}>0$. Besides, $\psi(u_{0})\le\psi(0)=0$. Then, $\psi(u)>\frac{\alpha}{8\rho_{0}^{2}}>\psi(0)\ge\psi(u_{0})$ for any $u\in\partial B_{\rho_{0}}$. So, $\psi(u_{0})<inf\{\psi(u):u\in\partial B_{\rho_{0}}\}$. Hence, $\psi$ has a local minimum $u_{0}\in B_{\rho_{0}}=\{u\in H_{0,\Delta}^{1}(J):\Arrowvert u\Arrowvert<\rho_{0}\}$
	
	Next, we will show that there exists $u_{1}$ with $\Arrowvert u_{1}\Arrowvert>\rho_{0}$ such that $\psi(u_{1})<inf\{\psi(u):u\in\partial B_{\rho_{0}}\}$.

From (\ref{F2}) and (\ref{I2}), 
	$$\begin{array}{lcl}
	\psi(u)&\le&\displaystyle\frac{\beta}{2}\Arrowvert u\Arrowvert^{2}-\sum_{j=1}^{p}e_{g}(t,0)\int_{0}^{u(t_{j})}I_{j}(t)dt-\int_{0}^{T}e_{g}(t,0)F(t,u^{\sigma}(t))\Delta t\\
	
&\le&\displaystyle\frac{\beta}{2}\Arrowvert u\Arrowvert^{2}-\sum_{j=1}^{p}e_{g}(t,0)\left(a_{1j}\arrowvert u(t_{j})\arrowvert^{\eta}-a_{2j}\right)-\int_{0}^{T}e_{g}(t,0)\left(a_{1}\arrowvert u^{\sigma}(t)\arrowvert^{\eta}-a_{2}\right)\\
	
&\le&\displaystyle\frac{\beta}{2}\Arrowvert u\Arrowvert^{2}-m\sum_{j=1}^{p}a_{1j}\arrowvert u(t_{j})\arrowvert^{\eta}+M\sum_{j=1}^{p}a_{2j}-ma_{1}\int_{0}^{T}\arrowvert u^{\sigma}(t)\arrowvert^{\eta}\Delta t+Ma_{2}T.
\end{array}$$
	
	Thus,
	\begin{displaymath}
	\psi(u)\le\frac{\beta}{2}\Arrowvert u\Arrowvert^{2}-m\sum_{j=1}^{p}a_{1j}\arrowvert u(t_{j})\arrowvert^{\eta}+M\sum_{j=1}^{p}a_{2j}-ma_{1}\Arrowvert u^{\sigma}\Arrowvert_{L_{\Delta}^{\eta}}^{\eta}+Ma_{2}T.
	\end{displaymath}
	
	For any $u\in H_{0,\Delta}^{1}(J)$ with $\Arrowvert u\Arrowvert=1$, we have 
	\begin{displaymath}
	\psi(Nu)\le\frac{\beta}{2}N^{2}-m\sum_{j=1}^{p}a_{1j}N^{\eta}\arrowvert u(t_{j})\arrowvert^{\eta}+M\sum_{j=1}^{p}a_{2j}-ma_{1}N^{\eta}\Arrowvert u^{\sigma}\Arrowvert_{L_{\Delta}^{\eta}}+Ma_{2}T, 
	\end{displaymath}
so, $\lim_{N\to\infty}\psi(Nu)=-\infty$
	since $\eta>2$. Then, there exists $N_{0}>\rho_{0}$, such that $\psi(N_{0}u)\le 0$.
	
	Hence, for the above $\rho_{0}$, there exists $u_{1}$ such that $\Arrowvert u_{1}\Arrowvert=N_{0}$ and $\psi(u_{1})<0$. 
	
	Then, we have $\max\{\psi(u_{0}),\psi(u_{1})\}<inf\{\psi(u):u\in\partial B_{T_{0}}\}$ 
	
	The next step is to show that $\psi$ satisfies the Palais-Smale condition.
	
	Let $\{\psi(u_{k})\}$ be a bounded sequence such that $\lim_{k\to\infty}\psi'(u_{k})=0$. Now we show that $\Arrowvert u_{k}\Arrowvert$ is bounded. By (\ref{psi}) we have
	
	\begin{equation}
	\label{uk}
\begin{array}{lcl}
	(\psi'(u_{k}),u_{k})&=&A(u_{k},u_{k})-\int_{0}^{T}e_{g}(t,0)f(t,u_{k}^{\sigma}(t))u_{k}^{\sigma}(t)\Delta t\\
\\
&&-\sum_{j=1}^{p}e_{g}(t,0)I_{j}(u_{k}(t_{j}))u_{k}(t_{j}).
\end{array}
	\end{equation}
Thus,
$$
\begin{array}{lcl}	
\displaystyle\psi(u_{k})-\frac{1}{\eta}(\psi'(u_{k}),u_{k})&=&\displaystyle\left(\frac{1}{2}-\frac{1}{\eta}\right)A(u_{k},u_{k})+\Gamma_{1}+\Gamma_{2}\\

&\ge&\displaystyle \left(\frac{1}{2}-\frac{1}{\eta}\right)\alpha\Arrowvert u_{k}\Arrowvert^{2}+\Gamma_{1}+\Gamma_{2},
\end{array}
$$
where 
	\begin{displaymath}
	\Gamma_{1}=\frac{1}{\eta}\int_{0}^{T}e_{g}(t,0)f(t,u_{k}^{\sigma}(t))u_{k}^{\sigma}(t)\Delta t-\int_{0}^{T}e_{g}(t,0)F(t,u_{k}^{\sigma}(t))\Delta t,
	\end{displaymath}
	\begin{displaymath}	\Gamma_{2}=\frac{1}{\eta}\sum_{j=1}^{p}e_{g}(t,0)I_{j}(u_{k}(t_{j}))u_{k}(t_{j})-\sum_{j=1}^{p}e_{g}(t,0)\int_{0}^{u_{k}(t_{j})}I_{j}(t)dt.
	\end{displaymath}
	
	Note that $J=\Upsilon_{1}^{k}\cup\Upsilon_{2}^{k}$, where $\Upsilon_{1}^{k}=\{t\in J: \arrowvert u_{k}(t)\arrowvert<\gamma\}$, $\Upsilon_{2}^{k}=\{t\in J:\arrowvert u_{k}(t)\arrowvert\ge\gamma\}$, and that there exists a constant $c$, such that
	\begin{equation}
	\label{Ffc}
	\arrowvert F(t,u^{\sigma}(t))\arrowvert\le c,\quad\arrowvert f(t,u^{\sigma}(t)u^{\sigma}(t))\arrowvert\le c,\quad\textrm{si}\quad\arrowvert u\arrowvert<\gamma,
	\end{equation}
	\begin{equation}
	\label{Ic}
	\arrowvert\int_{0}^{u(t_{j})}I_{j}(t)dt\arrowvert c,\quad\arrowvert I_{j}(u(t_{j}))u(t_{j})\arrowvert\le c,\quad\textrm{si}\quad\arrowvert u\arrowvert<\gamma.
	\end{equation}
	
	So by $(H_{1})$ and (\ref{Ffc}), we have
	
$$
\begin{array}{lcl}	
	\Gamma_{1}&=&\displaystyle \frac{1}{\eta}\int_{\Upsilon_{1}^{k}}e_{g}(t,0)f(t,u_{k}^{\sigma}(t))u_{k}^{\sigma}(t)\Delta t+\frac{1}{\eta}\int_{\Upsilon_{2}^{k}}e_{g}(t,0)f(t,u_{k}^{\sigma}(t))u_{k}^{\sigma}(t)\Delta t\\
	
&&-\displaystyle\int_{\Upsilon_{1}^{k}}e_{g}(t,0)F(t,u_{k}^{\sigma}(t))\Delta t-\int_{\Upsilon_{2}^{k}}e_{g}(t,0)F(t,u_{k}^{\sigma}(t))\Delta t\\
	
&\ge&\displaystyle -\frac{1}{\eta}\int_{\Upsilon_{1}^{k}}e_{g}(t,0)\arrowvert f(t,u_{k}^{\sigma}(t))u_{k}^{\sigma}(t)\arrowvert\Delta t-\int_{\Upsilon_{1}^{k}}e_{g}(t,0)\arrowvert F(t,u_{k}^{\sigma}(t))\arrowvert\Delta t\\
	
&&+\displaystyle\frac{1}{\eta}\int_{\Upsilon_{2}^{k}}e_{g}(t,0)f(t,u_{k}^{\sigma}(t))u_{k}^{\sigma}(t)\Delta t-\int_{\Upsilon_{2}^{k}}e_{g}(t,0) F(t,u_{k}^{\sigma}(t))\Delta t\\
	
	
	
	
&\ge& c'+c''=c_{1},
\end{array}
$$
	
	where $c'$, $c''$ and $c_{1}$ are constants (independent of $k$).
	
	Analogously, there exist a constant $c_{2}$ (independent of $k$), such that $\Upsilon_{2}^{k}\ge c_{2}$.

	Hence,
	$$
\begin{array}{lcl}
\psi(u_{k})&\ge& \displaystyle \left(\frac{1}{2}-\frac{1}{\eta}\right)\alpha\Arrowvert u_{k}\Arrowvert^{2}+\frac{1}{\eta}\left(\psi'(u_{k}),u_{k}\right)+\Gamma_{1}+\Gamma_{2}\\
&\ge&\displaystyle \left(\frac{1}{2}-\frac{1}{\eta}\right)\alpha\Arrowvert u_{k}\Arrowvert^{2}-\frac{1}{\eta}\Arrowvert\psi'(u_{k})\Arrowvert\Arrowvert u_{k}\Arrowvert+c_{1}+c_{2}
\end{array}
$$
	
	Since $\psi(u_{k})$ is bounded, we have $\{\Arrowvert u_{k}\Arrowvert\}_{k=1}^{\infty}$ is bounded sequence.
	
	Hence, there exists a subsequence $\{u_{k}\}$ (for simplicity denoted again by $\{u_{k}\}$) such that $\{u_{k}\}$ weakly converges to some $u$ in $H_{0,\Delta}^{1}(J)$. Then the sequence $\{u_{k}\}$ converges uniformly to $u$ in $\mathcal{C}(J)$.
	
	By (\ref{uk}), we have
	$$
\begin{array}{lcl}
\Arrowvert u_{k}\Arrowvert^{2}=&&\int_{0}^{T}e_{g}(t,0)\left(f(t,u_{k}^{\sigma})\ u_{k}^{\sigma}-\lambda \ u_{k}^{2}\right)\Delta t+(\psi'(u_{k}),u_{k})\\
\\
&+&\sum_{j=1}^{p}e_{g}(t,0)I_{j}(u_{k}(t_{j}))\ u_{k}(t_{j}).
\end{array}
$$

	So, we have 
	$$
\begin{array}{c}
\lim_{k\to\infty}\Arrowvert u_{k}\Arrowvert^{2}=\int_{0}^{T}e_{g}(t,0)\left(f(t,u^{\sigma})u^{\sigma}-\lambda u^{2}\right)\Delta t+\sum_{j=1}^{p}e_{g}(t,0)I_{j}(u(t_{j}))u(t_{j}).
\end{array}
$$

	Then $\Arrowvert u_{k}\Arrowvert$ converges in $H_{0,\Delta}^{1}(J)$. Since $H_{0,\Delta}^{1}(J)$ is a Hilbert space, and the sequence $\{u_{k}\}\in H_{0,\Delta}^{1}(J)$ satisfies $u_{k}\rightharpoonup u$, then $\{u_{k}\}$ converges to $u$, i.e., $u_{k}\to u$. $\psi$ satisfies the Palais-Smale condition.
	
	Now, by \ref{mpt}, there exists a critical point $u^{*}$. Therefore, $u_{0}$ and $u^{*}$ are two critical points of $\psi$, and they are classical solutions of (NPE). Hence, $u_{0}$ and $u^{*}$ are classical solutions of (NP). \qed
	
	\begin{ej}
	Let $\T=h\Z$ for $0<h<T$, $T\in h\Z$ and $t_{1}\in (0,T)$. Thus, $J=\lbrack 0,T\rbrack\cap h\Z$ and $J^{\kappa}=\lbrack 0,T-h\rbrack\cap h\Z$.
	Consider the following boundary value problem:
\begin{equation}
\label{example}
\left\{\begin{array}{lcl}
 -u^{\Delta \Delta}(t)+\alpha u^{\Delta}(t+h)+\lambda u(t+h)=tu^{5}(t+h);& &\textrm{$\Delta$-a.e.}\quad t\in J^{\kappa^{2}}\\
\noalign{\bigskip}  -(u^{\Delta}(t_{1}^{+})-u^{\Delta}(t_{1}^{-}))=u^{5}(t_{1})\\
\noalign{\bigskip}  u(0)=0=u(T).
\end{array}
\right .
\end{equation}

where $\alpha>0$ be constant.

We can see that $g(t)=\alpha$ is regressive and continuous. If we take $\eta=s=6$ and $\lambda>\frac{-\lambda_{1}}{(1+\alpha h)^{\frac{T}{h}}}$, by theorem \ref{teoprin}, the eq. (\ref{example}) has at least two solutions.
	\end{ej}
	
	\end{proof}
	
\bibliography{mibiblio}
\bibliographystyle{unsrt}

\end{document}